\documentclass[12pt]{amsart}
\usepackage[margin=1.2in]{geometry}
\usepackage{graphicx,latexsym}
\usepackage{amsfonts, amssymb, amsmath, amsthm}

\newtheorem{theorem}{Theorem}[section]
\newtheorem{proposition}[theorem]{Proposition}
\newtheorem{lemma}[theorem]{Lemma}
\newtheorem{corollary}[theorem]{Corollary}

\theoremstyle{definition}
\newtheorem{definition} [theorem]{Definition}

\theoremstyle{remark}
\newtheorem{remark}[theorem]{Remark}

 \numberwithin{equation}{section}

\begin{document}

\title[The STFT of distributions of exponential type]{The short-time Fourier transform of distributions of exponential type and Tauberian theorems for shift-asymptotics}

\author[S. Kostadinova]{Sanja Kostadinova}
\address{Faculty of Electrical Engineering and Information Technologies, Ss. Cyril and Methodius University, Rugjer Boshkovik bb, 1000 Skopje, Macedonia}
\email{ksanja@feit.ukim.edu.mk}

\author[S. Pilipovi\'{c}]{Stevan Pilipovi\'{c}}
\address{Department of Mathematics and Informatics, University of Novi Sad, Trg Dositeja Obradovi\'ca 4, 21000 Novi Sad, Serbia}
\email {stevan.pilipovic@dmi.uns.ac.rs}

\author[K. Saneva]{Katerina Saneva}
\address{Faculty of Electrical
Engineering and Information Technologies, Ss. Cyril
and Methodius University, Rugjer Boshkovik bb, 1000 Skopje, Macedonia}
\email{saneva@feit.ukim.edu.mk}

\author[J. Vindas]{Jasson Vindas}
\address{Department of Mathematics, Ghent University, Krijgslaan 281 Gebouw S22, 9000 Gent, Belgium}
\email{jvindas@cage.Ugent.be}

\subjclass[2010]{Primary 81S30, 40E05; Secondary 26A12, 41A60, 46F05, 46F12}
\keywords{short-time Fourier transform; Tauberian theorems; modulation spaces; $\mathcal {K}'_1$ and $\mathcal{K}_{1}$; distributions of exponential type; Silva tempered ultradistributions (tempered ultra-hyperfunctions); S-asymptotics; regularly varying functions}

\thanks{Research supported by the project 174024 of the Ministry of Education and Sciences of Serbia. S. Pilipovi\'{c} and J. Vindas also acknowledge support from Ghent University, through the BOF-grant number 01T00513.}

\begin{abstract}
We study the short-time Fourier transform on the space $\mathcal{K}_{1}'(\mathbb{R}^n)$ of distributions of exponential type. We give characterizations of  $\mathcal{K}_{1}'(\mathbb{R}^n)$ and some of its subspaces in terms of modulation spaces. We also obtain various Tauberian theorems for the short-time Fourier transform.
\end{abstract}

\maketitle


\section{Introduction}
The short-time Fourier transform (STFT) is a very effective device in the study of function spaces. The investigation of major test function spaces and their duals through time-frequency representations has attracted much attention. For example, the Schwartz class $\mathcal{S}(\mathbb{R}^{n})$ and the space of tempered distributions $\mathcal{S}'(\mathbb{R}^{n})$ were studied in \cite{GZ1} (cf. \cite{gr01}). Characterizations of Gelfand-Shilov spaces and ultradistribution spaces by means of the short-time Fourier transform and modulation spaces are also known \cite{GZ2,OMA,toft} (cf. \cite{cordero,c-p-r-t}).

The purpose of this paper is two folded. On the one hand we study the short-time Fourier transform in the context of the space $\mathcal{K}'_{1}(\mathbb{R}^{n})$ of distributions of exponential type, the dual of the space of exponentially rapidly decreasing smooth functions $\mathcal{K}_{1}(\mathbb{R}^{n})$ (see Section \ref{notations} for the definition of all spaces employed in this article). We will obtain various characterizations of $\mathcal{K}'_{1}(\mathbb{R}^{n})$ and related spaces via the short-time Fourier transform. The space $\mathcal{K}'_{1}(\mathbb{R}^{n})$  was introduced by Silva \cite{SS} and Hasumi \cite{hasumi} in connection with the so-called space of Silva tempered ultradistributions $\mathcal{U}'(\mathbb{C}^{n})$. Let us mention that $\mathcal{K}_{1}'(\mathbb{R}^{n})$ and $\mathcal{U}'(\mathbb{R}^{n})$ were also studied by Morimoto through the theory of ultra-hyperfunctions \cite{Morimoto} (cf. \cite{PM}). We refer to \cite{estrada-vindasB,hoskins-pinto,SS2,Z} for some applications of the Silva spaces. Our second goal is to present a new kind of Tauberian theorems. In such theorems the exponential asymptotics of functions and distributions can be obtained from those of the short-time Fourier transform. Let us state a sample of our results. In the next statement $L$ stands for a locally bounded Karamata slowly varying function \cite{BGT,Korevaarbook}, namely, a positive function that is asymptotically self-similar in the sense:
$$
\lim_{x\to\infty} \frac{L(a x)}{L(x)}= 1, \  \  \ \forall a>0.
$$

\begin{theorem}\label{STFT th1} Let $f$ be a positive non-decreasing function on $[0,\infty)$ and let $\psi$ be a positive function such that $\psi''\in L^{1}_{loc}(\mathbb{R})$ and $\int_{-\infty}^{\infty}(\psi(t)+ |\psi'(t)|+|\psi'' (t)|)e^{\beta t+\varepsilon|t|} dt <\infty$, where $\beta\geq0$ and $\varepsilon>0$. Suppose that the limits
\begin{equation}
\label{eq1}\lim_{x\to\infty}\frac{e^{2\pi i \xi x}}{e^{\beta x}L(e^{x})} \int _{0}^{\infty} f(t)\psi(t-x)e^{-2\pi i\xi t} \ dt= J(\xi)
\end{equation} exist for every $\xi\in\mathbb{R}$, then
\begin{equation}
\label{eq2}
\lim_{x\to\infty} \frac{f(x)}{e^{\beta x}L(e^{x})}= \frac{J(0)}{\int_{-\infty}^{\infty}\psi(t)e^{\beta t} dt}\ .
\end{equation}
Furthermore, if $L$ satisfies $L(xy)\leq A L(x)L(y)$ for all $x,y>0$ and some constant $A$, the requirements over $\psi$ can be relaxed to $\int_{-\infty}^{\infty}(\psi (t)+|\psi' (t)|+|\psi'' (t)|)L(e^{|t|})e^{\beta t} dt<\infty$.
\end{theorem}

It turns out that Theorem \ref{STFT th1} can be deduced from a more general type of Tauberian theorems. In Section \ref{s-asym} we  shall give precise descriptions of the S-asymptotic properties \cite{PSV} of a distribution in terms of the asymptotic behavior of its short-time Fourier transform (S-asymptotics stands for shift-asymptotics). The notion of $S$-asymptotics measures the asymptotic behavior of the translates of a distribution $T_{-h}f$ with respect the parameter $h$ and it is closely related to the extension of Wiener's Tauberian ideas \cite{Korevaarbook,Wiener} to the context of Schwartz distributions \cite{pilipovic93}. We would like to point out that there is an extensive literature about Tauberian theorems for Schwartz distributions, see, e.g., the monographs \cite{PSV,VDZ} and references therein. We also mention the article \cite{SAS}, where Tauberian theorems for the short-time Fourier transform were also studied, though with a different approach.

The plan of this article is as follows. In Section \ref{STFT} we shall present continuity theorems for the STFT and its adjoint on the test function space $\mathcal{K}_{1}(\mathbb{R}^{n})$ and the topological tensor product $\mathcal K_{1}({\mathbb R^n})\widehat{\otimes} \mathcal U(\mathbb C^n)$, where $\mathcal U(\mathbb C^n)$ is the space of entire rapidly decreasing functions in any horizontal band of $\mathbb C^n$. We then use such continuity results to develop a framework for the STFT on $\mathcal{K}_{1}'(\mathbb{R}^{n})$. We also introduce in this paper the space $\mathcal{B}'_{\omega}(\mathbb{R}^{n})$ of $\omega$-bounded distributions and its subspace $\dot{\mathcal{B}}'_{\omega}(\mathbb{R}^{n})$ with respect to an exponentially moderate weight $\omega$; when $\omega=1$, these spaces coincide with the well-known Schwartz spaces \cite[p. 200]{Schwartz} of bounded distributions $\mathcal{B}'(\mathbb{R}^{n})$ and $\dot{\mathcal{B}}'(\mathbb{R}^{n})$, which are of great importance in the study of convolution and growth properties of distributions. Notice that the distribution space $\mathcal{B}'(\mathbb{R}^{n})$ also plays an important role in Tauberian theory; see, for instance, Beurling's theorem \cite[p. 230]{donoghue} and the distributional Wiener Tauberian theorem from \cite{pilipovic93}.  The spaces $\mathcal{B}'_{\omega}(\mathbb{R}^{n})$ and $\dot{\mathcal{B}}'_{\omega}(\mathbb{R}^{n})$ will be characterized in Section \ref{characterization omega bounded} in terms of the short-time Fourier transform and also in terms properties of the set of translates of their elements.
Section \ref{modulation spaces} is devoted to the characterization of $\mathcal{K}_{1}'(\mathbb{R}^{n})$ and related spaces via modulation spaces. The conclusive Section \ref{s-asym} deals with Tauberian theorems, where in particular we give a proof of Theorem \ref{STFT th1}. Our Tauberian hypotheses are actually in terms of membership to suitable modulation spaces, this allows us to reinterpret the $S$-asymptotics in the weak$^{\ast}$ topology of modulation spaces.

\section{{Preliminaries}}\label{notations}


\subsection{Notation} We use the constants in the Fourier transform as
\begin{equation}\label{peq1}\mathcal{F}(\varphi)(\xi)=\widehat {\varphi}(\xi)=\int_{\mathbb R ^n} e^{-2\pi i x
\cdot\xi}\varphi(x)dx .
\end{equation}
The translation and modulation operators are defined by $T_{x}f(\: \cdot \:)=f( \: \cdot \:-x)$ and $M_{\xi}f(\: \cdot \:)=e^{2\pi i \xi
\:\cdot\:}f(\:\cdot\:),$ $ x,\xi\in\mathbb R ^n. $ The operators  $M_{\xi} T_x$ and $T_x M_{\xi}$ are called time-frequency shifts and we have $M_{\xi} T_x=e^{2\pi i x\cdot \xi}T_x M_{\xi}.$ The notation $\langle f,\varphi\rangle$ means dual pairing whereas $(f,\varphi)_{L^{2}}$ stands for the $L^{2}$ inner product. All dual spaces in this article are equipped with the strong dual topology. We denote by $\check{f}$ the function (or distribution) $\check{f}(t)=f(-t)$.

\subsection{The STFT} The short-time Fourier transform (STFT) of a
function $f\in L^{2}(\mathbb R ^n)$ with respect to a window function $\psi\in L^{2}(\mathbb R ^n)$  is defined as \begin{equation} \label{defSTFT} V_{\psi} f(x,\xi )   = \langle f, \overline{M_{\xi}T_x\psi} \rangle
=\int _{{\mathbb R}^n} f(t)\overline{\psi(t-x)}e^{-2\pi i\xi\cdot t} \ dt,\, \, x,\xi \in {\mathbb R}^n. \end{equation}
There holds
$\left\| V_{\psi} f\right\| _{2} =\left\| f\right\| _{2}\left\| \psi\right\| _{2}$. The adjoint of $V_{\psi}$ is given by the mapping
\begin{equation*}
V_{\psi}^{*}F(t)=\iint_{\mathbb R^{2n}}F(x, \xi)\psi(t-x)e^{2\pi i\xi\cdot t}dxd\xi,\end{equation*}
interpreted as an $L^{2}(\mathbb{R}^{n})$-valued weak integral. If $\psi\neq 0$ and
$\gamma\in L^{2}(\mathbb R ^n) $ is a  synthesis window for $\psi$, namely, $(\gamma,\psi )_{L^{2}} \neq 0$,
then for any $f\in L^2(\mathbb R^n),$
 \begin{equation} \label{inverzna} f=\frac{1}{( \gamma
,\psi )_{L^2} } \iint \nolimits _{{\mathbb R}^{2n} } V_{\psi} f(x,\xi ) M_{\xi } T_{x} \gamma d\xi dx . \end{equation}

Whenever the dual pairing in \eqref{defSTFT} is well-defined, the
definition of $V_{\psi} f$ can be generalized for $f$ in larger classes than $L^{2}(\mathbb{R}^{n})$, for instance: $f\in \mathcal D '({\mathbb R}^{n})$ and $\psi\in \mathcal D ({\mathbb R}^{n})$. In fact,
it is enough to have $\psi\in \mathcal{A}({\mathbb R}^{n})$ and $f\in \mathcal{A}'(\mathbb{R}^{n})$, where $\mathcal{A}({\mathbb R}^{n})$ is a time-frequency shift invariant topological vector space. Note also that the inversion formula \eqref{inverzna} holds pointwise when $f$ is sufficiently regular, for instance, for function in the Schwartz class $\mathcal{S}(\mathbb{R}^{n})$. For a complete account on the STFT, we refer to \cite{gr01}.

\subsection{Spaces}\label{spaces}
\noindent The Hasumi-Silva  \cite{SS,hasumi} test function space ${\mathcal K}_1(\mathbb R ^n)$ consists of those  $\varphi\in C^{\infty}(\mathbb{R}^{n})$ for which all norms
$$\nu_k(\varphi):=\sup_{t\in {\mathbb R  ^n}, \ |\alpha|\leq k}e^{k|t|}|\varphi^{(\alpha)}(t)|, \ \ \ k\in{\mathbb N}_0,$$
are finite. The
elements of ${\mathcal K}_1(\mathbb R ^n)$ are called exponentially rapidly decreasing smooth functions. It is easy to see that ${\mathcal K}_1(\mathbb R ^n)$ is an FS-space and therefore Montel and reflexive. The space ${\mathcal K}_1(\mathbb R ^n)$ is also nuclear \cite{hasumi}.

Note that if $\varphi\in {\mathcal K}_1(\mathbb R ^n)$, then the Fourier transform (\ref{peq1}) extends to an entire function. In fact, the Fourier transform is a topological isomorphism from ${\mathcal K}_1(\mathbb R ^n)$ onto $\mathcal U(\mathbb C ^n)$, the space of entire functions which decrease faster than any polynomial in bands. More precisely, a entire function $\phi\in\mathcal{U}(\mathbb{C}^{n})$ if and only if
$$
\dot{\nu}_{k}(\phi):=\sup_{z \in \Pi_{k}}(1+|z|^{2})^{k/2}|\phi(z)|<\infty, \ \ \forall k\in{\mathbb N}_{0},
$$
where
$\Pi_{k}$ is the tube $\Pi_{k}= \mathbb{R}^{n}+i[-k,k]^{n}.$

The dual space  ${\mathcal K}'_1({\mathbb R}^n)$ consists of all distributions $f$ of exponential type, i.e., those of the form
${f=\sum_{|\alpha|\leq l}(e^{s |\: \cdot\: |}f_\alpha)^{(\alpha)}}$, where $f_\alpha\in L^{\infty}(\mathbb{R}^{n})$ \cite{hasumi}. The Fourier transform extends to a topological isomorphism $\mathcal{F}:{\mathcal K}'_1({\mathbb R}^n)\to \mathcal{U}'(\mathbb{C}^{n})$, the latter space is known as the space of Silva tempered ultradistributions \cite{hasumi} (also called the space of tempered ultra-hyperfunctions \cite{Morimoto}). The space $\mathcal{U}'(\mathbb{C}^{n})$ contains the space of analytic functionals. See also the textbook \cite{hoskins-pinto} for more information about these spaces.

We introduce a generalization of the Schwartz space of bounded distributions $\mathcal{B}'(\mathbb{R}^{n})$ \cite[p. 200]{Schwartz}. Let $\omega:\mathbb{R}^{n}\to(0,\infty)$ be an exponentially moderate weight, namely, $\omega$ is measurable and satisfies the estimate
\begin{equation}
\label{weight omega}
\omega(x+y)\leq A \omega(y) e^{a|x|}, \  \  \  x,y\in\mathbb{R}^{n} ,
\end{equation}
for some constants $A>0$ and $a\geq 0$. For instance, any positive measurable function  $\omega$ which is submultiplicative, i.e., $\omega(x+y)\leq \omega(x)\omega(y)$, and integrable near the origin must necessarily satisfy (\ref{weight omega}), as follows from the standard results about subadditive functions \cite{beurling,hille-phillips}. Extending the Schwartz space $\mathcal{D}_{L^{1}}(\mathbb{R}^{n})$, we define the Fr\'{e}chet space $\mathcal{D}_{L^{1}_{\omega}}(\mathbb{R}^{n})=\{\varphi\in C^{\infty}(\mathbb{R}^{n}):\: \varphi^{(\alpha)}\in L^{1}_{\omega}(\mathbb{R}^{n}), \forall \alpha\in\mathbb{N}_{0}^{n}\}$, provided with the family of norms
$$
\|\varphi\|_{1,\omega,k}:=\sup_{|\alpha|\leq k} \int_{\mathbb{R}^{n}} |\varphi^{(\alpha)}(t)|\omega(t)dt, \ \ \ k\in\mathbb{N}_{0}.
$$
Then, $\mathcal{B}'_{\omega}(\mathbb{R}^{n})$ stands for the strong dual of $\mathcal{D}_{L^{1}_{\omega}}(\mathbb{R}^{n})$, i.e., $\mathcal{B}'_{\omega}(\mathbb{R}^{n})=(\mathcal{D}_{L^{1}_{\omega}}(\mathbb{R}^{n}))'$. Since we have the dense embedding $\mathcal{K}_{1}(\mathbb{R}^{n})\hookrightarrow\mathcal{D}_{L^{1}_{\omega}}(\mathbb{R}^{n})$, we have $\mathcal{B}'_{\omega}(\mathbb{R}^{n})\subset \mathcal{K}'_{1} (\mathbb{R}^{n})$. We call
$\mathcal{B}'_{\omega}(\mathbb{R}^{n})$ the space of $\omega$-bounded distributions. We also define $\dot{\mathcal{B}}'_{\omega}(\mathbb{R}^{n})$ as the closure of $\mathcal{D}(\mathbb{R}^{n})$ in $\mathcal{B}'_{\omega}(\mathbb{R}^{n})$.

Next, we shall consider $\mathcal K_{1}({\mathbb R}^n)\widehat{\otimes}\mathcal U(\mathbb C^n)$, the topological tensor product space obtained as the completion of
$\mathcal K_{1}({\mathbb R}^n)\otimes\mathcal U(\mathbb C^n)$  in, say, the $\pi$- or the $\varepsilon$- topology \cite{treves}. Explicitly, the nuclearity of $\mathcal{K}_{1}(\mathbb{R}^{n})$ implies that $\mathcal K_{1}({\mathbb R}^n)\widehat{\otimes}\mathcal U(\mathbb C^n)=\mathcal K_{1}({\mathbb R}^n)\widehat{\otimes}_{\pi}\mathcal U(\mathbb C^n)=\mathcal K_{1}({\mathbb R}^n)\widehat{\otimes}_{\varepsilon}\mathcal U(\mathbb C^n)$. Thus, the topology of
$\mathcal K_{1}({\mathbb R}^n)\widehat{\otimes} \mathcal U(\mathbb C^n)$ is given by the family of the norms
$$\rho_{k}(\Phi):=\sup_{(x,z)\in {\mathbb R}^n \times \Pi_{k},\ |\alpha|\leq k}e^{k|x|}(1+|z|^{2}) ^{k/2}\left|\frac{\partial^{\alpha}} {\partial x^{\alpha}}\Phi(x,
z)\right|, \ \ k\in{\mathbb N}_0,$$
and we also obtain $(\mathcal K_{1}({\mathbb R}^n)\widehat\otimes \mathcal
U(\mathbb C^n) )'=\mathcal K'_{1}({\mathbb R}^n)\widehat\otimes \mathcal U'(\mathbb C^n)$.

Finally, let $m$ be a weight on $\mathbb{R}^{2n}$, that is, $m:\mathbb{R}^{2n}\to (0,\infty)$ is measurable and locally bounded. Then, if $p,q\in [1,\infty]$, the weighted Banach space ${L_{m}^{p,q}}(\mathbb R^{2n})$ consists of all measurable functions $F$ such that
$$
\|F\|_{L_{m}^{p,q}}:=\left(\int_{\mathbb R^n}\left(\int_{\mathbb R^n}|F(x,\xi)|^{p}m(x,\xi)^{p} dx\right)^{q/p}d\xi\right)^{1/q}<\infty.$$
(With the obvious modification when $p=\infty$ or $ q=\infty$.)


\section{Short-time Fourier transform of distributions of exponential type}\label{STFT}
In this section we study the mapping properties of the STFT on the space of distributions of exponential type. Note that the STFT extends to the sesquilinear mapping $(f,\psi)\mapsto V_{\psi}f$ and its adjoint induces the bilinear mapping $(F,\psi)\mapsto V^{\ast}_{\psi}F$.

We start with the test function space $\mathcal{K}_{1}(\mathbb{R}^{n})$. If $f,\psi \in \mathcal{K}_{1}(\mathbb{R}^{n})$, then we immediately get that (\ref{defSTFT}) extends to a holomorphic function in the second variable, namely, $V_{\psi}f(x,z)$ is entire in $z\in \mathbb{C}^{n}$. We write in the sequel $z=\xi+i\eta$ with $\xi,\eta\in \mathbb{R}^{n}$. Observe also that an application of the Cauchy theorem shows that if $\Phi\in \mathcal K_{1}({\mathbb R^n})\widehat{\otimes} \mathcal U(\mathbb C^n)$ and $\psi\in \mathcal{K}_{1}(\mathbb{R}^{n})$, then for arbitrary $\eta\in\mathbb{R}^{n}$ we may write $V^{\ast}_{\psi}\Phi$ as
\begin{equation}\label{adjoint} V_{\psi}^{*}\Phi(t) = \iint_{\mathbb R ^{2n}}\Phi(x, \xi+i\eta)\psi(t-x) e^{2\pi i(\xi+i\eta) \cdot t}dxd\xi.
\end{equation}

Our first proposition deals with the range and continuity properties of $V$ and $V^{\ast}$ on test function spaces.

\begin{proposition}\label{neprekinatSTFT} The following mappings are continuous:
\begin{itemize}
\item [$(i)$] $V:\mathcal K_{1}({\mathbb R}^n)\times \mathcal K_{1}({\mathbb R}^n)\to \mathcal
K_{1}({\mathbb R}^n)\widehat{\otimes} \mathcal U(\mathbb C ^n)$.
\item [$(ii)$] $V^{*}:(\mathcal K_{1}({\mathbb R^n})\widehat{\otimes} \mathcal U(\mathbb C^n))\times\mathcal K_{1}(\mathbb R^n)\to \mathcal K_{1}({\mathbb R^n}).
$
\end{itemize}
\end{proposition}

\begin{proof} For part $(i)$, let $\varphi,\psi\in \mathcal{K}_{1}(\mathbb{R}^{n})$. Let $k$ be an even integer. If $(x,z)\in \mathbb{R}^{n}\times\Pi_{k}$ and $|\alpha|\leq k$, then

\begin{align*}
&e^{k|x|}(1+|z|^{2}) ^{k/2}\left|\frac{\partial^{\alpha}}{\partial x^{\alpha}}V_{\psi}\varphi (x, z)\right|
\\
&
\leq (1+nk^{2})^{k/2}e^{k|x|} \left|\int_{\mathbb{R}^{n}} (1-\Delta_{t})^{k/2}(\varphi(t)\overline{\psi^{(\alpha)}(t-x)}e^{2\pi \eta\cdot t}) dt\right|\\
&
\leq \tilde{C}_{k}\sum_{|\beta_{1}|+|\beta_{2}|\leq k} e^{k|x|}\int_{\mathbb R^n} \left|\varphi^{(\beta_{1})}(t)\overline{\psi^{(\alpha+\beta_{2})}(t-x)}\right |e^{2 \pi k|t|}dt,
\end{align*}
which shows that $\rho_{k}(V_{\psi}\varphi)\leq C_{k} \nu_{8k}(\varphi)\nu_{k}(\psi)$. For $(ii)$, if $\Phi\in \mathcal K_{1}({\mathbb R^n})\widehat{\otimes} \mathcal U(\mathbb C^n)$, $\psi\in \mathcal{K}_{1}(\mathbb{R}^{n})$, and $|\alpha|\leq k$, we obtain
\begin{align*}
e^{k|t|}\left|\frac{\partial^{\alpha}}{\partial t^{\alpha}}V_{\psi}^{*}\Phi(t)\right|&\leq (2\pi)^{|\alpha|}\sum_{\beta\leq \alpha} \binom{\alpha}{\beta}e^{k|t|}\iint_{\mathbb{R}^{2n}}|\xi|^{k}|\Phi(x, \xi)||\psi^{(\beta)}(t-x)|dxd\xi\\
&
\leq (4\pi)^{|\alpha|}\nu_{k}(\psi) \iint_{\mathbb{R}^{2n}}|\xi|^{k}e^{k|x|}|\Phi(x, \xi)|dxd\xi
\\
&
\leq A_{k,n}\nu_{k}(\psi)\rho_{k+n+1}(\Phi);
\end{align*}
hence $\rho_{k}(V^{\ast}_{\psi}\Phi)\leq A_{k,n}\nu_{k}(\psi)\rho_{k+n+1}(\Phi)$.

\end{proof}

Observe that if the window $\psi\in\mathcal{K}_{1}(\mathbb{R}^{n})\setminus\{0\}$ and  $\gamma\in \mathcal{K}_{1}(\mathbb{R}^{n})$ is a synthesis window, the reconstruction formula (\ref{inverzna}) reads as:
\begin{equation}
\label{reconstructiontest}
\frac{1}{( \gamma
,\psi )_{L^{2}} } V_{\gamma}^* V_\psi={\rm id}_{\mathcal K_{1}({\mathbb R^n})}.
\end{equation}

We now study the STFT on $\mathcal K_1'(\mathbb R^n)$. Notice that the modulation operators $M_{z}$ operate continuously on $\mathcal{K}_{1}(\mathbb{R}^{n})$ even when $z\in\mathbb{C}^{n}$. Thus, if $f\in\mathcal{K}'_{1}(\mathbb{R}^{n})$ and $\psi\in\mathcal{K}_{1}(\mathbb{R}^{n})$ then $V_{\psi}f$, defined by the dual pairing in (\ref{defSTFT}), also extends in the second variable as an entire function $V_{\psi}f(x,z)$ in $z\in\mathbb{C}^{n}$. Furthermore, it is clear that $V_{\psi}f(x,z)$ is $C^{\infty}$ in $x\in\mathbb{R}^{n}$. We begin with a lemma.

\begin{lemma} \label{lemmaExchange}Let $\psi\in\mathcal{K}_{1}(\mathbb{R}^{n})$.
\begin{enumerate}
\item [$(a)$] Let $B'\subset \mathcal{K}'_{1}(\mathbb{R}^{n})$ be a bounded set. There is $k=k_{B'}\in\mathbb{N}_{0}$ such that
\begin{equation}\label{kprime}\sup_{f\in B',\:(x,z)\in\mathbb{R}^{n}\times\Pi_{\lambda}}e^{-k|x|- 2\pi x\cdot \Im m\: z}(1+|z|)^{-k}|V_\psi f(x,z)|<\infty,  \  \  \ \forall{\lambda}\geq 0.
\end{equation}
\item [$(b)$] For every $f\in\mathcal{K}'_{1}(\mathbb{R}^{n})$ and $\Phi\in \mathcal K_{1}({\mathbb R^n})\widehat{\otimes} \mathcal U(\mathbb C^n)$,

\begin{equation}\label{def11}\langle V_\psi{f}, \Phi \rangle =\left\langle f,
\overline{V_{\psi}^*\overline{\Phi}}\right\rangle.\end{equation}
\end{enumerate}
\end{lemma}

\begin{proof} \emph{Part $(a)$}. By the Banach-Steinhaus theorem, $B'$ is equicontinuous, so that there are $C>0$ and $k\in\mathbb{N}_{0}$ such that $|\langle f,\varphi \rangle|\leq C\nu_{k}(\varphi)$, $\forall f\in B', \forall\varphi\in \mathcal{K}_{1}(\mathbb{R}^{n})$. Hence, for all $f \in  B'$ ($z=\xi+i\eta$),
\begin{align*}
\left|V_{\psi}f(x, z)\right|&\leq C\sup_{t \in \mathbb R ^n, |\alpha|\leq k}e^{k|t|}\left|\frac{\partial^{\alpha}}{\partial t^{\alpha}}\left(e^{-2 \pi i z\cdot  t}\overline{\psi(t-x)}\right)\right|
\\
&
\leq (2\pi)^{k} C (1+|z|^{2})^{k/2}\sup_{t \in \mathbb R ^n,
|\alpha|\leq k}e^{k|t|+2\pi\eta\cdot t}\sum_{\beta\leq \alpha}\binom{\alpha}{\beta}\left|\psi ^{(\beta)}(t-x)\right|
\\
&
\leq
(4\pi)^{k}C (1+|z|^{2})^{k/2} e^{k|x|+2\pi\eta \cdot x}\nu_{k+1+\left\lfloor 2\pi|\eta|\right\rfloor}(\psi),
\end{align*}
where $\left\lfloor 2\pi |\eta|\right\rfloor$ stands for the integral part of $2\pi |\eta|$.

\emph{Part $(b)$}. We first remark that the left hand side of (\ref{def11}) is well defined because of part $(a)$. To show (\ref{def11}), notice that  the integral in (\ref{adjoint}), with $\eta=0$, can be approximated by a sequence of convergent Riemann sums in the topology of $\mathcal{K}_{1}(\mathbb{R}^{n})$; this justifies the exchange of integral and dual pairing in
$$
 \left\langle f(t),\iint_{\mathbb R ^{2n}}\Phi(x, \xi)e^{-2 \pi i \xi\cdot t}\overline{\psi(t-x)} dxd\xi\right\rangle_{t}= \iint_{\mathbb R ^{2n}}\Phi(x, \xi) \langle f, \overline{M_{\xi}T_x\psi} \rangle dxd\xi,
$$
which is the same as (\ref{def11}).
\end{proof}

In particular, if $B'$ is a singleton, part (a) of Lemma \ref{lemmaExchange} gives the growth order of the function $V_{\psi}f$ on every set $\mathbb{R}^{n}\times \Pi_{\lambda}$.

Let us define the adjoint STFT on $\in\mathcal K_{1}'({\mathbb R^n})\widehat{\otimes} \mathcal U'(\mathbb C^n)$.

\begin{definition} Let $\psi\in\mathcal K_{1}({\mathbb R^n})$. The adjoint STFT $V^*_{\psi} $ of $F\in\mathcal K_{1}'({\mathbb R^n})\widehat{\otimes} \mathcal U'(\mathbb C^n)$
is the distribution $V_{\psi}^*F\in \mathcal{K}'_{1}(\mathbb{R}^{n})$ whose action on test functions is given by
\begin{equation}\label{def1}\langle V_{\psi}^*F, \varphi\rangle:=\left\langle F, \overline{V_{\psi}\overline{\varphi}}\right\rangle, \ \ \ \varphi \in \mathcal K_{1}({\mathbb R^n}).\end{equation} \end{definition}

The next theorem summarizes our results.

\begin{theorem}
\label{thSTFT} The two STFT mappings
\begin{enumerate}
\item[(i)] $V:\mathcal K'_{1}({\mathbb R}^n)\times \mathcal K_{1}({\mathbb R}^n)\to \mathcal
K'_{1}({\mathbb R}^n)\widehat{\otimes} \mathcal U'(\mathbb C ^n)$
\item [(ii)]$V^{*}:(\mathcal K'_{1}({\mathbb R^n})\widehat{\otimes} \mathcal U'(\mathbb C^n))\times\mathcal K_{1}(\mathbb R^n)\to \mathcal K'_{1}({\mathbb R^n})$
\end{enumerate}
 are hypocontinuous. Let $\psi\in \mathcal{K}_1(\mathbb{R}^{n})\setminus\{0\}$ and let $\gamma\in \mathcal{K}_1(\mathbb{R}^{n})$ be a synthesis window for it. The following inversion and desingularization formulas hold:
\begin{equation}
\label{reconstructiondistributions}
\frac{1}{( \gamma
,\psi )_{L^{2}} } V_{\gamma}^* V_\psi={\rm id}_{\mathcal K'_{1}({\mathbb R^n})},
\end{equation}
and, for all $f \in \mathcal{K}'_1(\mathbb{R}^{n})$, $\varphi\in\mathcal{K}_1(\mathbb{R}^{n}),$ and $\eta\in\mathbb{R}^{n}$,
\begin{equation}
\label{desingular}
\langle f, \varphi\rangle=\frac{1}{( \gamma
,\psi )_{L^{2}} }\iint_{\mathbb{R}^{2n}}V_{\psi}f(x,\xi+i\eta)V_{\overline{\gamma}}\varphi (x,-\xi-i\eta)  dxd\xi.
\end{equation}
\end{theorem}
\begin{proof} That $V$ and $V^{\ast}$ are hypocontinuous on these spaces follows from Proposition \ref{neprekinatSTFT} and the formula (\ref{def11}) from Lemma \ref{lemmaExchange}; we leave the details to the reader. By the Cauchy theorem, it is enough to show (\ref{desingular}) for $\eta=0$. Using (\ref{def1}), (\ref{def11}), and (\ref{reconstructiontest}), we have $\langle V^{\ast}_{\gamma}V_{\psi}f,\varphi\rangle=\langle V_{\psi}f,\overline{V_{\gamma}\overline{\varphi}}\rangle=\langle f,\overline{V^{\ast}_{\psi}V_{\gamma}\overline{\varphi}}\rangle= ( \gamma
,\psi )_{L^{2}} \langle f, \varphi\rangle,$ namely, (\ref{reconstructiondistributions}) and (\ref{desingular}).\end{proof}

The next corollary gives the converse to part $(a)$ of Lemma \ref{lemmaExchange} under a weaker inequality than (\ref{kprime}), namely, a characterization of bounded sets in  $\mathcal K_{1}'(\mathbb R^n)$ in terms of the STFT.

\begin{corollary} \label{boundedsets}Let $B'\subset
\mathcal K_{1}'(\mathbb R^n)$ and $\psi\in\mathcal K_{1}(\mathbb R^n)\setminus\{0\}$. If there are $\eta\in \mathbb{R}^{n}$ and $k\in\mathbb{N}_{0}$ such that
\begin{equation}\label{bounded}\sup_{f\in B',(x,\xi)\in\mathbb {R}^{2n}}e^{-k|x|}(1+|\xi|)^{-k}|V_\psi f(x,\xi+i\eta)|<\infty,
\end{equation}
then the set $B'$ is bounded in $\mathcal K_{1}'(\mathbb R^n)$. Conversely, if $B'$ is bounded in $\mathcal K_{1}'(\mathbb R^n)$ there is $k\in\mathbb{N}_{0}$ such that (\ref{kprime}) holds.  \end{corollary}

\begin{proof}
In view of the Banach-Steinhaus theorem, we only need to show that $B'$ is weakly bounded. Let $\gamma$ be a synthesis window for $\psi$ and let $\varphi \in \mathcal K_{1}(\mathbb R ^n).$ Then, by the desingularization formula \eqref{desingular}, we have
\begin{equation*}\sup_{f\in B'}\left|\left\langle f, \varphi\right\rangle\right|
\leq \frac{C_{\eta}}{(\gamma,\psi )_{L^{2}}}\iint_{\mathbb
R^{2n}}e^{k|x|}(1+|\xi|)^k\left|V_{\overline{\gamma}}\varphi(x, -\xi-i\eta)\right|dxd\xi<\infty,
\end{equation*}
because $V_{\gamma}\varphi\in\mathcal K_{1}({\mathbb R^n})\widehat{\otimes} \mathcal U(\mathbb C^n)$. The converse was already shown in Lemma \ref{lemmaExchange}.
\end{proof}

\section{Characterizations of $\mathcal{B}'_{\omega}(\mathbb{R}^{n})$ and $\dot{\mathcal{B}}'_{\omega}(\mathbb{R}^{n})$}
\label{characterization omega bounded}
We now turn our attention to the characterization of the space of $\omega$-bounded distributions $\mathcal{B}'_{\omega}(\mathbb{R}^{n})$ and its subspace $\dot{\mathcal{B}}'_{\omega}(\mathbb{R}^{n})$. Recall that $\omega$ stands for an exponentially moderate weight, i.e., a positive and measurable function satisfying (\ref{weight omega}).
\begin{theorem}\label{omega bounded} Let $f\in\mathcal{K}'_{1}(\mathbb{R}^{n})$ and $\psi\in \mathcal{K}_{1}(\mathbb{R}^{n})\setminus\{0\}$.
\begin{itemize}
\item [$(i)$]The following statements are equivalent:
\begin{itemize}
\item [$(a)$] $f\in\mathcal{B}'_{\omega} (\mathbb{R}^{n}).$
\item [$(b)$] The set $\{T_{-h}f/\omega(h):\: h\in\mathbb{R}^{n}\}$ is bounded in $\mathcal{K}_{1}'(\mathbb{R}^{n})$.
\item [$(c)$] There is $s\in\mathbb{R}$ such that
\begin{equation}\label{inequality omega modulation}
\sup_{(x,\xi)\in\mathbb {R}^{2n}}(1+|\xi|)^{-s}\frac{|V_\psi f(x,\xi)|}{\omega(x)}<\infty.
\end{equation}
\end{itemize}
\item[$(ii)$] The next three conditions are equivalent:
\begin{itemize}
\item [$(a)'$] $f\in\dot{\mathcal{B}}'_{\omega} (\mathbb{R}^{n}).$
\item [$(b)'$] $\lim_{|h|\to\infty}T_{-h}f/\omega(h)=0$ in $\mathcal{K}_{1}'(\mathbb{R}^{n})$.
\item [$(c)'$] There is $s'\in\mathbb{R}$ such that
\begin{equation}\label{limit omega modulation}
\lim_{|(x,\xi)|\to\infty}(1+|\xi|)^{-s'}\frac{|V_\psi f(x,\xi)|}{\omega(x)}=0.
\end{equation}
\end{itemize}
\end{itemize}
\end{theorem}
\begin{remark} \label{rk4.1} Theorem \ref{omega bounded} remains valid if we replace $\mathcal{K}'_{1}(\mathbb{R}^{n})$ and $\mathcal{K}_{1}(\mathbb{R}^{n})$ by $\mathcal{D}'(\mathbb{R}^{n})$ and $\mathcal{D}(\mathbb{R}^{n})$ everywhere in the statement. Schwartz has shown in \cite[p. 204]{Schwartz} the equivalence between $(a)$ and $(b)$ for $\omega=1$ by using a much more complicated method involving a parametrix technique.
\end{remark}

\begin{proof} \emph{Part $(i)$.} $(a)\Rightarrow(b)$. Let $f\in\mathcal{B}'_{\omega} (\mathbb{R}^{n})$, since $\mathcal{K}_{1}(\mathbb{R}^{n})$ is barreled, we only need to show that the $f\ast \varphi$ is bounded by $\omega$ for fixed $\varphi\in \mathcal{K}_{1}(\mathbb{R}^{n})$. Let $B:=\{\phi\in \mathcal{D}(\mathbb{R}^{n}):\: \int_{\mathbb{R}^{n}}|\phi(x)|\omega(x)dx\leq 1\}$. By the assumption (\ref{weight omega}),

$$\|\check{\varphi}\ast\phi\|_{1,\omega,k}\leq A\max_{|\alpha|\leq k}\int_{\mathbb{R}^{n}}|\varphi^{(\alpha)}(x)|e^{a|x|}dx, \  \  \ \forall k\in\mathbb{N}_{0},\:\forall \phi\in B, $$
namely, the set $\check{\varphi}\ast B$ is bounded in $\mathcal{D}_{L^{1}_{\omega}}(\mathbb{R}^{n})$. Consequently, $\sup_{\phi\in B}|\langle f\ast\varphi, \phi\rangle|=\sup_{\phi\in B}|\langle f,\check{\varphi} \ast \phi\rangle|<\infty$. Since $\mathcal{D}(\mathbb{R}^{n})$ is dense in $L^{1}_{\omega}(\mathbb{R}^{n})$, this implies that $f\ast \varphi \in (L^{1}_{\omega}(\mathbb{R}^{n}))'$, i.e., $\sup_{h\in\mathbb{R}^{n}} |(f\ast\varphi)(h)|/\omega(h)<\infty$, as claimed.

$(b)\Rightarrow(c)$. Notice that $(V_{\psi}T_{-h}f)(x,z)=e^{2\pi i z\cdot h}V_{\psi}f(x+h,z)$. Fix $\lambda\geq 0$. By Corollary \ref{boundedsets} (cf. (\ref{kprime})), there are $k\in\mathbb{N}_{0}$ and $C_{\lambda}>0$ such that, for all $x,h\in\mathbb{R}^{n}$ and $z\in \Pi_{\lambda}$,
$$
|e^{2\pi i z\cdot h}V_{\psi} f(x+h,z)|\leq C_{\lambda}\omega(h)(1+|z|)^{k}e^{(k+2\pi\lambda)|x|}.
$$
Taking $x=0$ and $\Im m \: z=0$, one gets (\ref{inequality omega modulation}).

$(c)\Rightarrow(a)$. Fix a synthesis window $\gamma\in \mathcal{K}_{1}(\mathbb{R}^{n})$. In view of (\ref{weight omega}), one has that if $j$ is any non-negative even integer and $\lambda\geq0$, then, for all $\varphi\in \mathcal{D}_{L_{\omega}^{1}}(\mathbb{R}^{n})$,
\begin{align*}
&\sup_{z\in\Pi_{\lambda}}(1+|z|^{2})^{j/2}\int_{\mathbb{R}^{n}}e^{-2\pi x\cdot \Im m\: z }\omega(x)|V_{\overline{\gamma}}\varphi(x,z)|dx
\\
&
\leq \tilde{C}_{j} \sum_{|\beta_{1}|+|\beta_2|\leq j} \iint_{\mathbb{R}^{2n}} \omega(x) |\varphi^{(\beta_{1})}(t)\gamma^{(\beta_{2})} (t-x)|e^{2\pi \lambda |t-x|} dtdx
\\
&
\leq A\tilde{C}_{j}\|\varphi\|_{1,\omega,j} \max_{|\beta|\leq j}\int_{\mathbb{R}^{n}}|\gamma^{(\beta)}(x)|e^{(2\pi\lambda+a)|x|}dx\leq C_{j,\lambda} \|\varphi\|_{1,\omega,j}.
\end{align*}
We may assume that $s$ is an even integer. By (\ref{inequality omega modulation}) and the previous estimate, we obtain, for every $\varphi\in \mathcal{K}_{1}(\mathbb{R}^{n})$,
\begin{align*}
|\langle f,\varphi\rangle|&\leq \frac{C}{( \psi,\gamma)_{L^{2}}}\iint_{\mathbb
R^{2n}}(1+|\xi|)^s\omega(x)\left|V_{\overline{\gamma}}\varphi(x, -\xi)\right|dxd\xi\leq C_{s} \|\varphi\|_{1,\omega,s+n+1},
\end{align*}
which yields $f\in\mathcal{B}'_{\omega}(\mathbb{R}^{n})$.

\emph{Part $(ii)$.} Any of the conditions implies that $f\in \mathcal{B}'_{\omega}(\mathbb{R}^{n})$. $(a)'\Rightarrow(b)'$. Fix $\varphi\in \mathcal{K}_{1}(\mathbb{R}^{n})$. Given fixed $\varepsilon>0$, we must show that $\limsup_{|h|\to\infty}|\langle T_{-h}f,\varphi\rangle|/\omega(h)\leq \varepsilon$. Notice  that $\{T_{h}\varphi/\omega(h):\: h\in\mathbb{R}^{n}\}$ is a bounded set in $\mathcal{D}_{L^{1}_{\omega}}(\mathbb{R}^{n})$. Since $f$ is in the closure of $\mathcal{D}(\mathbb{R}^{n})$ in $\mathcal{B}'_{\omega}(\mathbb{R}^{n})$, there is $\phi\in \mathcal{D}(\mathbb{R}^{n})$ such that $|\langle T_{-h}(f-\phi), \varphi|\leq \varepsilon \omega(h)$ for every $h\in\mathbb{R}^{n}$. Consequently,
$$
\limsup_{|h|\to\infty}\frac{|\langle T_{-h}f,\varphi\rangle|}{\omega(h)}\leq \varepsilon + \lim_{|h|\to\infty}\frac{1}{\omega(h)}\left|\int_{\mathbb{R}^{n}} \varphi(t-h)\phi(t)dt\right|\leq \varepsilon.
$$

$(b)'\Rightarrow(c)'$. If $\xi$ remains on a compact of $K\subset\mathbb{R}^{n}$, then $\{\overline{M_{\xi}\psi}: \: \xi\in K\}$ is compact in $\mathcal{K}_{1}(\mathbb{R}^{n})$, thus, by the Banach-Steinhaus theorem,
$$0=\lim_{|x|\to\infty} \frac{|\langle T_{-x}f,\overline{M_{\xi}\psi}\rangle|}{\omega(x)}= \lim_{|x|\to\infty} \frac{|V_{\psi}f(x,\xi)|}{\omega(x)} , \mbox{ uniformly in }\xi\in K.$$
There is $s$ such that (\ref{inequality omega modulation}) holds. Taking into account that the above limit holds for arbitrary $K$, we obtain that (\ref{limit omega modulation}) is satisfied for any $s'>s$.

$(c)'\Rightarrow(a)'$. We may assume that $s'$ is a non-negative even integer. Consider the weight ${\omega_{s'}(x,\xi)=\omega(x)(1+|\xi|)^{s'}}$. The limit relation (\ref{limit omega modulation}) implies that $V_{\psi}f$ is in the closure of $\mathcal{K}_{1}(\mathbb{R}^{n})\otimes \mathcal{S}(\mathbb{R}^{n})$ with respect to the norm $\|\: \:\|_{L^{\infty,\infty}_{1/\omega_{s'}}}$. Since we have the dense embedding $\mathcal{U}(\mathbb{C}^{n})\hookrightarrow\mathcal{S}(\mathbb{R}^{n})$, there is a sequence $\{\Phi_{j}\}_{j=1}^{\infty}\subset \mathcal{K}_{1}(\mathbb{R}^{n})\widehat{\otimes}\mathcal{U}(\mathbb{C}^{n})$ such that $\lim_{j\to\infty}\Phi_{j}=V_{\psi}f$ in $L^{\infty,\infty}_{1/\omega_{s'}}(\mathbb{R}^{2n})$. Let $\gamma\in\mathcal{K}_{1}(\mathbb{R}^{n})$ be a synthesis window and set $\phi_{j}=V^{\ast}_{\gamma}\Phi_{j}\in\mathcal{K}_{1}(\mathbb{R}^{n})$ (cf. Proposition \ref{neprekinatSTFT}). By the relations (\ref{desingular}) and (\ref{def1}), we have for any $\varphi\in \mathcal{K}_{1}(\mathbb{R}^{n})$,
$$
|\langle f-\phi_{j},\varphi\rangle|\leq \frac{C \|\varphi\|_{1,\omega,s+n+1}}{(\gamma,\psi )_{L^{2}}} \|V_{\psi}f-\Phi_{j}\|_{L^{\infty,\infty}_{1/\omega_{s'}}},
$$
where $C$ does not depend on $j$. Thus, $\phi_{j}\to f$ in $\mathcal{B}'_{\omega}(\mathbb{R}^{n})$, which in turn implies that $f\in\dot{\mathcal{B}}'_{\omega}(\mathbb{R}^{n})$ because $\mathcal{D}(\mathbb{R}^{n})\hookrightarrow \mathcal{K}_{1}(\mathbb{R}^{n})$.

\end{proof}

We immediately get the ensuing result, a corollary of Theorem \ref{omega bounded}.

\begin{corollary}
$
\mathcal{K}'_{1}(\mathbb{R}^{n})=\bigcup_{\omega} \mathcal{B}'_{\omega}(\mathbb{R}^{n})=\bigcup_{\omega}\dot{\mathcal{B}}'_{\omega}(\mathbb{R}^{n}).
$
In particular, $f\in\mathcal{D}'(\mathbb{R}^{n})$ belongs to $\mathcal{K}'_{1}(\mathbb{R}^{n})$ if and only if there is $s\in\mathbb{R}$ such that $\{e^{-s|h|}T_{-h}f: \:h\in\mathbb{R}^{n}\}$ is bounded in $\mathcal{D}'(\mathbb{R}^{n})$.
\end{corollary}
\section{Characterizations through modulation spaces}\label{modulation spaces}
We present here the characterization of the spaces $\mathcal{K}_{1}(\mathbb{R}^{n})$, $\mathcal{K}_{1}'(\mathbb{R}^{n})$, $\mathcal{B}'_{\omega}(\mathbb{R}^{n})$, $\dot{\mathcal{B}}'_{\omega}(\mathbb{R}^{n})$, $\mathcal{U}(\mathbb{C}^{n})$, and $\mathcal{U}'(\mathbb{C}^{n})$ in terms of modulation spaces.

Let us recall the definition of the modulation spaces. There are several equivalent ways to introduce them \cite{gr01}. Here we follow the approach from \cite{cordero,c-p-r-t} based on Gelfand-Shilov spaces. We are interested in modulation spaces with respect to weights that are exponentially moderate. We denote by $\mathfrak{M}$ the class of all weight functions $m$ on $\mathbb{R}^{2n}$ that satisfy inequalities (for some constants $A>0$ and $a\geq0$):
$$
\frac{m(x_1+x_2,\xi_1+\xi_2)}{m(x_1,\xi_1)}\leq A e^{a(|x_2|+|\xi_2|)}, \  \  \ (x_1,\xi_1),(x_2,\xi_2)\in\mathbb{R}^{2n}.
$$
Observe that any so-called $v$-moderate weight \cite{gr01} belongs to $\mathfrak{M}$. We also consider the Gelfand-Shilov space
 $\Sigma_{1}^{1}(\mathbb{R}^{n})$ of Beurling type (sometimes also denoted as $\mathcal{S}^{(1)}(\mathbb{R}^{n})$ or $\mathcal{G}(\mathbb{R}^{n})$) and its dual $(\Sigma_{1}^{1})'(\mathbb{R}^{n})$. The space $\Sigma_{1}^{1}(\mathbb{R}^{n})$ consists \cite{CKL} of all entire functions $\varphi$ such that
$$
\sup_{x\in\mathbb{R}^{n}}|\varphi(x)|e^{\lambda |x|}<\infty \ \ \mbox{ and } \ \ \sup_{\xi\in\mathbb{R}^{n}}|\widehat{\varphi}(\xi)|e^{\lambda |\xi|}d\xi<\infty, \  \  \  \forall{\lambda>0}.
$$
We refer to \cite{pilipovic88} for topological properties of $\Sigma_{1}^{1}(\mathbb{R}^{n})$. The dual space $(\Sigma_{1}^{1})'(\mathbb{R}^{n})$ is also known as the space of Silva ultradistributions of exponential type \cite{hoskins-pinto,SS2} or the space of Fourier ultra-hyperfunctions \cite{PM}. If $m\in\mathfrak{M}$, $\psi\in\Sigma_{1}^{1}(\mathbb{R}^{n})\setminus\{0\}$, and $p, q \in[1,\infty]$, the modulation space $M_{m}^{p,q}(\mathbb{R}^{n})$ is defined as the Banach space
\begin{equation}
\label{modulation}
M_{m}^{p,q}(\mathbb R^n)=\{f \in \mathcal (\Sigma_{1}^{1})'(\mathbb R^n):\: \|f\|_{M_{m}^{p,q}}:=\|V_{\psi}f\|_{L_{m}^{p,q}}<\infty\}.
\end{equation}
This definition does not depend on the choice of the window $\psi$, as different windows lead to equivalent norms.
If $p=q$, then we write $M^{p}_{m}(\mathbb{R}^{n})$ instead of $M_{m}^{p, q}(\mathbb{R}^{n})$. The space $M^{1}_{m}(\mathbb{R}^{n})$ (for $m=1$) was original introduced by Feichtinger in \cite{fe81-3}. We shall also define $\dot{M}_{m}^{\infty}(\mathbb{R}^{n})$ as the closed subspace of $M^{\infty}_{m}(\mathbb{R}^{n})$ given by $\dot{M}_{m}^{\infty}(\mathbb{R}^{n})=\{f\in\mathcal (\Sigma_{1}^{1})'(\mathbb R^n):\: \lim_{|(x,\xi)|\to\infty}m(x,\xi)|V_{\psi}f(x,\xi)|=0\}$.

We now connect the space of exponential distributions with the modulation spaces. For it, we consider the weight subclass $\mathfrak{M}_{1}\subset \mathfrak{M}$ consisting of all weights $m$ such that (for some $s,a\geq 0$ and $A>0$)
\begin{equation}
\label{weight mod}
\frac{m(x_1+x_2,\xi_1+\xi_2)}{m(x_1,\xi_1)}\leq A e^{a|x_2|}(1+|\xi_{2}|)^{s}, \ \ \ (x_1,\xi_1),(x_2,\xi_2)\in\mathbb{R}^{2n}.
\end{equation}
Let $m\in\mathfrak{M}_{1}$. By Proposition \ref{neprekinatSTFT}, $\mathcal{K}_{1}(\mathbb{R}^{n})\subset M^{p,q}_{m}(\mathbb{R}^{n})$.  Since $\Sigma_{1}^{1}(\mathbb{R}^{n})\hookrightarrow \mathcal{K}_{1}(\mathbb{R}^{n})$, we obtain that $\mathcal{K}_{1}(\mathbb{R}^{n})$ is dense (weakly$^{\ast}$ dense if $p=\infty$ or $q=\infty$) in $M^{p,q}_{m}(\mathbb{R}^{n})$ and therefore $M^{p,q}_{m}(\mathbb{R}^{n})\subset \mathcal{K}'_{1}(\mathbb{R}^{n})$. It follows from the results of \cite{gr01} that we may use $\psi\in\mathcal{K}_{1}(\mathbb{R}^{n})\setminus\{0\}$ in (\ref{modulation}). Also, if $f\in M_{m}^{p,q}(\mathbb{R}^{n})$ and $\psi\in\mathcal{K}_{1}(\mathbb{R}^{n})$ then $V_{\psi}f$ is an entire function in the second variable (cf. Section \ref{STFT}); the next proposition describes the norm behavior of $V_{\psi}f(x,z)$ in the complex variable $z\in\mathbb{C}^{n}$.

\begin{proposition}\label{modulation complex variable}
Let $m\in \mathfrak{M}_{1}$, $p,q\in[1,\infty]$, and $\psi\in\mathcal{K}_{1}(\mathbb{R}^{n})\setminus\{0\}$. If $f\in M_{m}^{p,q}(\mathbb{R}^{n})$, then $(\forall\lambda\geq 0)$
\begin{equation}
\label{norm modulation complex}
\sup_{|\eta|\leq \lambda}\left(\int_{\mathbb R^n}\left(\int_{\mathbb R^n}|e^{-2\pi x\cdot \eta}V_\psi f(x,\xi+i\eta)m(x,\xi)|^{p} dx\right)^{q/p}d\xi\right)^{1/q}<C_{\lambda} \|f\|_{M^{p,q}_{m}}.
\end{equation}
(With obvious changes if $p=\infty$ or $q=\infty$.)
\end{proposition}
\begin{proof} Assume that $m$ satisfies (\ref{weight mod}) and set $v(x,\xi)=(1+|\xi|)^{s}e^{a|x|}$. Notice first that ${e^{-2\pi x\cdot \eta}V_\psi f(x,\xi+i\eta)= V_{\psi_{\eta}}f(x,\xi)}$, where $\psi_{\eta}(t)=e^{2\pi \eta\cdot t}\psi(t)$. As in the proof of \cite[Prop. 11.3.2, p. 234]{gr01},
$$
||V_{\psi_{\eta}}f||_{L^{p,q}_{m}}=\frac{1}{\|\psi\|^{2}_{L^{2}}}||(V_{\psi_\eta}V_\psi^*)V_\psi f||_{L^{p,q}_m}\leq C||V_{\psi_\eta}\psi||_{L^1_v}||V_{\psi}f||_{L^{p,q}_m}.
$$
Since $\{\psi_{\eta}: |\eta|\leq \lambda\}$ is bounded in $\mathcal{K}_{1}(\mathbb{R}^{n})$, we obtain that $\{V_{\psi_{\eta}}\psi:\:|\eta|\leq \lambda\}$ is bounded in $\mathcal{K}_{1}(\mathbb{R}^{n})\widehat{\otimes}\mathcal{U}(\mathbb{C}^{n})$; hence $\sup_{|\eta|\leq \lambda}||V_{\psi_\eta}\psi||_{L^1_v}<\infty$.

\end{proof}
Using the fundamental identity of time-frequency analysis, i.e. \cite[p. 40]{gr01}  $V_{\psi}f(x,\xi)= e^{-2\pi i x\cdot \xi}V_{\widehat{\psi}}\widehat{f}(\xi,-x)$, we can transfer results from $\mathcal{K}_{1}'(\mathbb{R}^{n})$ into $\mathcal{U}'(\mathbb{R}^{n})$ by employing the weight class $\mathfrak{M}_{2}=\{m\in\mathfrak{M}:\: \tilde{m}(x,\xi)= m(\xi,x)\in \mathfrak{M}_{1}\}$. For $s,a\geq0$, we employ the following special classes of weights ($\omega$ satisfies the conditions imposed in Subsection \ref{spaces}):
$$
v_{s,a}(x,\xi):=e^{a|x|}(1+|\xi|)^{s} \ \ \mbox{ and } \omega_{s}(x,\xi):=\omega(x) (1+|\xi|)^{s}.
$$
Clearly $v_{s,a},\omega_{s}\in \mathfrak{M}_{1}$. Obviously, for every $m\in\mathfrak{M}_{1}$ there are $s,a\geq 0$ such that ${M^{p,q}_{v_{s,a}}(\mathbb{R}^{n})\subseteq M^{p,q}_{m}(\mathbb{R}^{n})\subseteq M^{p,q}_{1/v_{s,a}}(\mathbb{R}^{n})}$.

\begin{proposition}\label{modulation th} Let $p,q\in[1,\infty]$. Then,

\begin{equation}\label{modspace1}\mathcal K_{1}'(\mathbb R^n)=\bigcup_{m\in\mathfrak{M}_{1}}M_{m}^{p,q}(\mathbb R^n), \  \  \ \ \ \mathcal{U}'(\mathbb C^n)=\bigcup_{m\in\mathfrak{M}_{2}}M_{m}^{p,q}(\mathbb R^n),\end{equation}

\begin{equation}\label{modspace2}
\mathcal K_{1}(\mathbb R^n)=\bigcap_{m\in\mathfrak{M}_{1}}M_{m}^{p,q}(\mathbb R^n), \  \  \ \ \ \mathcal{U}(\mathbb C^n)=\bigcap_{m\in\mathfrak{M}_{2}}M_{m}^{p,q}(\mathbb R^n),
  \end{equation}
	
	\begin{equation}\label{modspace3}
\mathcal {B} '_{\omega}(\mathbb R^n)=\bigcup_{s> 0}M_{1/\omega_{s}}^{\infty}(\mathbb R^n), \  \mbox{ and } \ \dot{\mathcal{B}}'_{\omega}(\mathbb R^n)=\bigcup_{s> 0}\dot{M}_{1/\omega_{s}}^{\infty}(\mathbb R^n).
  \end{equation}

 \end{proposition}

\begin{proof}
The results for $\mathcal{U}(\mathbb{C}^{n})$ and $\mathcal{U}'(\mathbb{C}^{n})$ follow from those for $\mathcal{K}_{1}(\mathbb{R}^{n})$ and $\mathcal{K}'_{1}(\mathbb{R}^{n})$.  The equalities in (\ref{modspace3}) are a reformulation of the equivalences $(a)\Leftrightarrow(c)$ and $(a)'\Leftrightarrow(c)'$ from Theorem \ref{omega bounded}. By (\ref{weight mod}) and \cite[Cor. 12.1.10, p. 254]{gr01}, given $m\in\mathfrak{M}_{1}$, there are $s,a>0 $ such that the embeddings $M^{\infty}_{v_{s+n+1,a+\varepsilon}}(\mathbb{R}^{n})\subseteq M^{p,q}_{m}(\mathbb{R}^{n})\subseteq M^{\infty}_{1/v_{s,a}}(\mathbb{R}^{n})$ hold. Thus, part $(a)$ from Lemma \ref{lemmaExchange} gives the equality $\mathcal{K}_{1}'(\mathbb{R}^{n})=\bigcup_{s,a>0}M^{\infty}_{1/v_{s,a}}(\mathbb{R}^{n})=\bigcup_{m\in\mathfrak{M}_{1}}M_{m}^{p,q}(\mathbb R^n).$ In view of Proposition \ref{neprekinatSTFT}, it only remains to show that
$$
\bigcap_{m\in\mathfrak{M}_{1}}M_{m}^{p,q}(\mathbb R^n)=\bigcap_{s,a>0}M^{\infty}_{v_{s,a}}(\mathbb{R}^{n})\subseteq \mathcal{K}_{1}(\mathbb{R}^{n}).
$$
We show the latter inclusion by proving that if $f\in M^{\infty}_{v_{s,a}}(\mathbb{R}^{n})$ (with $s,a>0$), then $\widehat{f}$ is holomorphic in the tube $\mathbb{R}^{n}+i\{\eta\in\mathbb{R}^{n}:\: |\eta|<a/(2\pi)\}$ and satisfies
\begin{equation}
\label{analytic bound}
\sup_{|\Im m\: z|\leq \lambda} (1+|z|^{2})^{s/2} |\widehat{f}(z)|<\infty, \  \  \ \forall\lambda<\frac{a}{2\pi}.
\end{equation}
In fact, choose a positive window $\psi\in \mathcal{D}(\mathbb{R}^{n})$ such that $\sum_{j\in \mathbb Z^n}\psi(t-j)=1$ for all $t\in\mathbb{R}^{n}$. Since $f=\sum_{j \in \mathbb {Z}^n}fT_{j}\psi$, we obtain $\widehat{f}=\sum_{j\in\mathbb{Z}^{n}}V_{\psi}f(j, \:\cdot\:)$, with convergence in $\mathcal{U}'(\mathbb{C}^{n})$. In view of Proposition \ref{modulation complex variable}, each $V_{\psi}f(j, z)$ is entire in $z$ and satisfies the bounds
$$
\sup_{|\Im m\: z|\leq \lambda}|(1+|z|^{2})^{s/2}|V_{\psi}(j,z)|< C_{\lambda } e^{-(a-2\pi\lambda)|j|}.
$$
The Weierstrass theorem implies that $\widehat{f}(z)=\sum_{j\in\mathbb{Z}^{n}}V_{\psi}f(j, z)$ is holomorphic in the stated tube domain and we also obtain (\ref{analytic bound}). Summing up, if $f\in\bigcap_{s,a>0}M^{\infty}_{v_{s,a}}(\mathbb{R}^{n})$, then $\widehat{f}\in\mathcal{U}(\mathbb{C}^{n})$, i.e., $f\in\mathcal{K}_{1}(\mathbb{R}^{n})$.
\end{proof}

The following corollary collects what was shown in the proof of Proposition \ref{modulation th}.
\begin{corollary} Let $s,a>0$. If $f\in M_{v_{s,a}}^{\infty}(\mathbb R^n)$, then $\widehat{f}$ is holomorphic in the tube $\mathbb{R}^{n}+i\{\eta\in\mathbb{R}^{n}:\: |\eta|<a/(2\pi)\}$ and satisfies the bounds (\ref{analytic bound}).
\end{corollary}

We make a remark concerning Proposition \ref{modulation th}.

\begin{remark}
\label{rk5} Employing \cite[Thrms. 3.2 and 3.4]{toft}, Proposition \ref{modulation th}
can be extended for $p,q\in (0,\infty].$
\end{remark}

\section{Tauberian theorems for S-asymptotics of distributions}\label{s-asym}

In this section we characterize the so-called $S$-asymptotic behavior of distributions in terms of the STFT. We briefly explain this notion;  we refer to \cite{PSV} for a complete treatment of the subject.

Let $f\in \mathcal{K}'_{1}(\mathbb{R}^{n})$. The idea of the $S$-asymptotics is to study the asymptotic properties of the translates $T_{-h}f$ with respect to a locally bounded and measurable comparison function  $c:\mathbb{R}^{n}\to (0,\infty)$. It is said that $f$ has $S$-asymptotic behavior with respect to $c$ if there is $g\in\mathcal{D}'(\mathbb{R}^{n})$ such that
\begin{equation}
\label{Seq1}\lim_{|h|\to\infty} \frac{1}{c(h)}T_{-h}f= g \  \  \ \mbox{in }  \mathcal{D}'(\mathbb{R}^{n}).
\end{equation}
The distribution $g$ is not arbitrary; in fact, one can show \cite{PSV} that the relation (\ref{Seq1}) forces it to have the form $g(t)=C e^{\beta \cdot t}$, for some $C\in\mathbb{R}$ and $\beta\in \mathbb{R}^{n}$. If $C\neq 0$, one can also prove \cite{PSV} that $c$ must satisfy the asymptotic relation
\begin{equation}
\label{Seq2}\lim_{|h|\to\infty} \frac{c(t+h)}{c(h)}= e^{\beta \cdot t},  \  \  \ \mbox{uniformly for }t \mbox{ in compact subsets of } \mathbb{R}^{n}.
\end{equation}
>From now on, we shall always assume that $c$ satisfies (\ref{Seq2}). A typical example of  such a $c$ is any function of the form $c(t)=e^{\beta \cdot t} L(e^{|t|})$, where $L$ is a Karamata slowly varying function \cite{BGT}. The assumption (\ref{Seq2}) implies \cite{PSV} that (\ref{Seq1}) actually holds in the space $\mathcal{K}'_{1}(\mathbb{R}^{n})$.  We will use the more suggestive notation
\begin{equation}
\label{Seq3}f(t+h)\sim c(h)g(t)  \  \  \ \mbox{ in } \mathcal{K}'_{1}(\mathbb{R}^{n})\  \mbox{ as }|h|\to\infty
\end{equation}
for denoting (\ref{Seq1}), which of course means that
$
(f\ast \check{\varphi})(h) \sim c(h)\int_{\mathbb{R}^{n}}\varphi(t) g(t) dt$ as $|h|\to\infty,$
for each $\varphi \in \mathcal{D}(\mathbb{R}^{n})$ (or, equivalently, $\varphi \in \mathcal{K}_1(\mathbb{R}^{n})$). In order to move further, we give an asymptotic representation formula and Potter type estimates \cite{BGT} for $c$:
\begin{lemma} \label{lemma c} The locally bounded measurable function $c$ satisfies (\ref{Seq2}) if and only if there is $b\in C^{\infty}(\mathbb{R}^{n})$ such that $\lim_{|x|\to \infty}b^{(\alpha)}(x)=0$ for every  multi-index $|\alpha|>0$
and
\begin{equation}\label{representation eq}
c(x)\sim \exp \left(\beta \cdot x+ b(x) \right) \ \ \ \mbox{ as }|x|\to\infty.
\end{equation}
In particular, for each $\varepsilon>0$ there are constants $a_{\varepsilon}, A_{\varepsilon}>0$ such that
\begin{equation}
\label{Seq5}
a_{\varepsilon}\exp(\beta\cdot t-\varepsilon |t|)\leq \frac{ c(t+h)}{c(h)}\leq A_{\varepsilon}\exp(\beta\cdot t+\varepsilon |t|), \ \ \  t, h\in\mathbb{R}^{n}.
\end{equation}
\end{lemma}
\begin{proof} By considering $e^{-\beta\cdot t}c(t)$, one may assume that $\beta=0$.
Let $\varphi\in \mathcal{D}(\mathbb{R}^{n})$ be such that $\int_{\mathbb{R}^{n}} \varphi(t)dt=1$. Set $b(x)= \int_{\mathbb{R}^{n}} \log c(t+x)\varphi(t)dt$. Clearly, $b\in C^{\infty}(\mathbb{R}^{n})$ and the relation (\ref{Seq2}) implies that $b(x)=\log c(x)+o(1)$ and $b^{(\alpha)}(x)=o(1)$ as $|x|\to\infty$, for each multi-index $|\alpha|>0$. This gives (\ref{representation eq}). Conversely, since $c$ is locally bounded, we may assume that actually $c(x)=e^{\beta\cdot x +b(x)}$, but $|b(t+h)-b(h)|\leq |t|\max_{\xi\in[h,t+h]}|\nabla b(\xi)|,$ which gives (\ref{Seq2}). Using the fact that $|\nabla b|$ is bounded, the same argument yields (\ref{Seq5}).
\end{proof}

Observe that Lemma \ref{lemma c} also tells us that the space $\mathcal{B}'_{c}(\mathbb{R}^{n})$ is well-defined for $c$. We can now characterize (\ref{Seq3}) in terms of the STFT. The direct part of the following theorem is an Abelian result, while the converse may be regarded as a Tauberian theorem.

\begin{theorem} \label{te1} Let $f\in\mathcal{K}_{1}'(\mathbb{R}^{n})$ and $\psi \in \mathcal K_{1}(\mathbb R^{n} )\backslash\{0\}$. If $f\in \mathcal K'_{1}(\mathbb {R}^{n} )$ has the S-asymptotic behavior (\ref{Seq3})
then, for every $\lambda\geq0$,
\begin{equation}\label{consab0}
\lim_{|h|\to\infty}e^{2 \pi i z\cdot h}\frac{V_{\psi} f(x+h,z)}{c(h)}= V_\psi g(x,z),
\end{equation}
uniformly for $z\in \Pi_{\lambda}$ and $x$ in compact subsets of $\mathbb R^{n}$.

Conversely, suppose that the limits

\begin{equation}\label{limit2} \lim_{|x|\to \infty }e^{2\pi i \xi\cdot x}\frac{V_{\psi}f(x,\xi)}{c(x)}=J(\xi)\in\mathbb{C}\end{equation}
exist for almost every $\xi\in\mathbb{R}^{n}$. If there is $s\in\mathbb{R}$ such that
\begin{equation}\label{Taubeq}
\sup_{(x,\xi)\in\mathbb {R}^{2n}}(1+|\xi|)^{-s}\frac{|V_\psi f(x,\xi)|}{c(x)}<\infty,
\end{equation}
then $f$ has the S-asymptotic behavior (\ref{Seq3}) with $g(t)=C e^{\beta \cdot t}$, where the constant is completely determined by the equation $J(\xi)=C \overline{\widehat{\psi}(-\xi+i\beta/(2\pi))}$.
\end{theorem}
\begin{remark}\label{rk5.1} Assume (\ref{Taubeq}). Consider a weight of the form $m_{\varepsilon}(x,\xi)=e^{\beta\cdot x+\varepsilon |x|}(1+|\xi|)^{s}$ with $\varepsilon>0$. It will be shown below that the asymptotics (\ref{Seq3}) holds in the weak$^{\ast}$ topology of $M^{\infty}_{1/m_{\varepsilon}}(\mathbb{R}^{n})$, i.e., $(f\ast \check{\varphi})(h)\sim c(h)\langle g, \varphi\rangle$ as $|h|\to\infty$ for every $\varphi$ in the modulation space $M_{m_{\varepsilon}}^{1}(\mathbb{R}^{n})$. Furthermore, one may use in (\ref{limit2}) and (\ref{Taubeq}) a window $\psi\in M_{m_{\varepsilon}}^{1}(\mathbb{R}^{n})\backslash\{0\}$.
\end{remark}
\begin{proof} Fix $\lambda\geq 0$ and a compact $K\subset \mathbb{R}^{n}$. Note that the set
$$\{\overline{M_{z}T_{x}\psi}: \: (x,z)\in K\times \Pi_{\lambda}\}$$
is compact in $\mathcal{K}_{1}(\mathbb{R}^{n})$. By the Banach-Steinhaus theorem,
$$\lim_{|h| \rightarrow \infty}e^{2 \pi i zh}\frac {V_{\psi}f(x+h,z)}{c(h)}
=\lim_{|h|\rightarrow \infty} \Big\langle\,\frac{T_{-h}f}{c(h)},\overline{M_{z}T_{x}\psi}\Big\rangle=
\Big\langle g,\overline{M_{z}T_{x}\psi}\Big\rangle,
$$
uniformly with respect to $(x,z)\in K\times \Pi_{\lambda}$, as asserted in (\ref{consab0}).

Conversely, assume (\ref{limit2}) and (\ref{Taubeq}). Let $H=\{\xi\in\mathbb{R}^{n}:\: \mbox{(\ref{limit2}) holds} \}$. In view of Theorem \ref{omega bounded}, we have that $f\in \mathcal{B}'_{c}(\mathbb{R}^{n})$ or, equivalently,  $\{T_{-h}f/c(h):\: h\in\mathbb{R}^{n}\}$ is bounded in $\mathcal{K}'_{1}(\mathbb{R}^{n})$. By the Banach-Steinhaus theorem and the Montel property of $\mathcal{K}'_{1}(\mathbb{R}^{n})$, $T_{-h}f/c(h)$ converges strongly to a distribution $g$ in $\mathcal{K}'_{1}(\mathbb{R}^{n})$ if and only if $\lim_{|h|\to \infty} \langle T_{-h}f,\varphi \rangle/c(h)$ exists for $\varphi$ in a dense subspace of $\mathcal{K}_{1}(\mathbb{R}^{n})$. Let $D$ be the linear span of $\{\overline{M_{\xi}T_{x}\psi}:\: (x,\xi)\in\mathbb{R}^{n}\times H\}$. By the desingularization formula (\ref{desingular}) and the Hahn-Banach theorem, we have that $D$ is dense in $\mathcal{K}_{1}(\mathbb{R}^{n})$. Thus, it suffices to verify that $\lim_{|h|\to\infty}\langle T_{-h}f,  \overline{M_{\xi}T_{x}\psi}\rangle/c(h)$ exists for each $(x,\xi)\in \mathbb{R}^{n}\times H$. But in this case (\ref{Seq2}) and (\ref{limit2}) yield
\begin{align*}
\lim_{|h|\to\infty}\frac{\langle T_{-h}f,  \overline{M_{\xi}T_{x}\psi}\rangle}{c(h)}&=\lim_{|h|\to\infty} e^{2\pi i \xi\cdot h}\frac{V_{\psi}f(x+h,\xi)}{c(h)}
\\
&
=e^{(\beta -2\pi i \xi)\cdot x}\lim_{|h|\to\infty} e^{2\pi i \xi\cdot (x+h)}\frac{V_{\psi}f(x+h,\xi)}{c(h+x)}
\\
&
= e^{(\beta -2\pi i \xi)\cdot x} J(\xi),
\end{align*}
as required. We already know that $g(t)=Ce^{\beta\cdot t}$. Comparison between (\ref{consab0}) and (\ref{limit2}) leads to $J(\xi)=V_{\psi}g(0,\xi)= C\int_{\mathbb{R}^{n}}\overline{\psi(t)} e^{\beta\cdot t-2\pi i\xi\cdot t}dt$. To show the assertion from Remark \ref{rk5.1}, note first that, by using (\ref{Seq5}), one readily verifies that
$$\sup_{h\in\mathbb{R}^{n}}\frac{ ||T_{-h}f||_{M^{\infty}_{1/m_{\varepsilon}}}}{c(h)}<\infty.$$
Since we have the dense embedding $\mathcal{K}_{1}(\mathbb{R}^{n})\hookrightarrow M^{1}_{m_{\varepsilon}}(\mathbb{R}^{n})$, we also have that $D$ is dense in $M^{1}_{m_{\varepsilon}}(\mathbb{R}^{n})$ and the assertion follows at once. The fact that one may use a window $\psi\in M_{m_{\varepsilon}}^{1}(\mathbb{R}^{n})\backslash\{0\}$ in (\ref{limit2}) and (\ref{Taubeq}) follows in a similar fashion because in this case the desingularization formula (\ref{desingular}) still holds.

\end{proof}

Let us make two addenda to Theorem \ref{te1}. The ensuing corollary improves Remark \ref{rk5.1}, provided that $c$ satisfies the extended submultiplicative condition (for some $A>0$):
\begin{equation}
\label{eq5.9} c(t+h)\leq A c(t)c(h).
\end{equation}
\begin{corollary} \label{TaubCor}Assume that $c$ satisfies (\ref{eq5.9}) and set $c_{s}(x,\xi)=c(x)(1+|\xi|)^{s}$, $s\in\mathbb{R}$. If $f\in M^{\infty}_{1/c_{s}}(\mathbb{R}^{n})$ and there is $\psi\in M^{1}_{c_{s}}(\mathbb{R}^{n})\setminus\{0\}$ such that the limits (\ref{limit2}) exist for almost every $\xi\in\mathbb{R}^{n}$, then, for some $g$, the $S$-asymptotic behavior (\ref{Seq3}) holds weakly$^{\ast}$ in $M^{\infty}_{1/c_{s}}(\mathbb{R}^{n})$, that is, $(f\ast \check{\varphi})(h)\sim c(h)\langle g, \varphi\rangle$ as $|h|\to\infty$ for every $\varphi\in M_{c_{s}}^{1}(\mathbb{R}^{n})$.
\end{corollary}
\begin{proof}
We retain the notation from the proof of Theorem \ref{te1}. The assumption $f\in M^{\infty}_{1/c_{s}}(\mathbb{R}^{n})$ of course tells us that (\ref{Taubeq}) holds. Employing the hypothesis (\ref{eq5.9}), one readily sees that $\sup_{h\in\mathbb{R}^{n}}\|T_{-h}f\|_{M^{\infty}_{1/c_{s}}}/c(h)<\infty$. A similar argument to the one used in the proof of Theorem \ref{te1} yields that the set $D$ associated to $\psi$ is dense in $M^{1}_{c_{s}}(\mathbb{R}^{n})$, which as above yields the result.
\end{proof}
In dimension $n=1$, the next theorem actually obtains the ordinary asymptotic behavior of $f$ in case it is a regular distribution on $(0,\infty)$ satisfying an additional Tauberian condition. We fix $m_{\varepsilon}$ as in Remark \ref{rk5.1} and $c_{s}$ as in Corollary \ref{TaubCor}.

\begin{theorem}
\label{te2} Let $f\in M_{1/c_{s}}^{\infty}(\mathbb{R})$. Suppose that
\begin{equation}\label{limit3}
\lim_{x\to \infty }e^{2\pi i \xi\cdot x}\frac{V_{\psi}f(x,\xi)}{c(x)}=J(\xi)\in\mathbb{C},
\end{equation}
for almost every $\xi\in \mathbb{R}$, where $\psi \in M^{1}_{m_{\varepsilon}}(\mathbb R)\setminus\{0\}$  (resp. $\psi \in M^{1}_{c_{s}}(\mathbb R)\setminus\{0\}$ if $c$ satisfies (\ref{eq5.9})). If there is $\alpha\geq 0$ such that $e^{\alpha t} f(t)$ is a positive non-decreasing function on the interval $(0,\infty)$, then
\begin{equation}
\label{asympeq}
\lim_{t\to\infty}\frac{f(t)}{c(t)}=C,
\end{equation}
where $C$ is the constant from Theorem \ref{te1}.
\end{theorem}
\begin{proof} Using (\ref{limit3}), the same method from Theorem \ref{te1} applies to show that $f(t+h)\sim C g(t)$ in $\mathcal{K}_{1}'(\mathbb{R})$ as $h\to\infty$, where $g(t)=Ce^{\beta t}$. We may assume that $\alpha\geq- \beta$. Set $\tilde{f}(t)=e^{\alpha t}f(t)$,  $b(t)=e^{\alpha t}c(t)$, and $r=\alpha+\beta\geq 0$. It is enough to show that $\tilde{f}(t)\sim C b(t)$ as $t\to\infty$, whence (\ref{asympeq}) would follow. By (\ref{Seq3}), we have that
\begin{equation*}
\tilde{f}(t+h)\sim b(h) Ce^{rt} \ \ \ \mbox{as }h\to\infty \mbox{ in } \mathcal{K}_{1}'(\mathbb{R}),
\end{equation*}
i.e.,
\begin{equation}
\label{asympeq2}
\langle \tilde{f}(t+h),\varphi(t)\rangle\sim C b(h)\int_{-\infty}^{\infty}e^{r t}\varphi(t)dt ,  \ \ \ \forall\varphi\in\mathcal{K}_{1}(\mathbb{R}^{n}).
\end{equation}
Let $\varepsilon>0$ be arbitrary. Choose a non-negative test function $\varphi\in\mathcal{D}(\mathbb{R})$ such that $\operatorname*{supp}\varphi\subseteq(0,\varepsilon)$ and $\int_{0}^{\varepsilon}\varphi(t)\mathrm{d}t=1$. Using the fact that $\tilde{f}$ is non-decreasing on $(0,\infty)$ and (\ref{asympeq2}), we obtain
\begin{align*}
\limsup_{h\to\infty}\frac{\tilde{f}(h)}{b(h)}&=\limsup_{h\to\infty} \frac{\tilde{f}(h)}{b(h)}\int_{0}^{\varepsilon}\varphi(t)dt\leq\lim_{h\to\infty}\frac{1}{b(h)}\int_{0}^{\varepsilon}\tilde{f}(t+h)\varphi(t)dt
\\
&
=\lim_{h\to\infty}\frac{\langle \tilde{f}(t+h),\varphi(t)\rangle}{b(h)}
=C \int_{0}^{\varepsilon}e^{r t}\varphi(t)\mathrm{d}t \leq C e^{r \varepsilon},
\end{align*}
taking $\varepsilon\to0^{+}$, we have shown that $\limsup_{h\to\infty}\tilde{f}(h)/b(h)\leq C$. Similarly, choosing in (\ref{asympeq2}) a non-negative $\varphi$ such that $\operatorname*{supp}\varphi\subseteq(-\varepsilon,0)$ and $\int_{-\varepsilon}^{0}\varphi(t)dt=1$, one obtains $\liminf_{h\to\infty}\tilde{f}(h)/b(h)\geq C$. This shows that $\tilde{f}(t)\sim C b(t)$ as $t\to\infty$, as claimed.
\end{proof}

We conclude this article with a proof of Theorem \ref{STFT th1}.

\begin{proof}[Proof of Theorem \ref{STFT th1}] Set $c(t)=e^{\beta t}L(e^{|t|})$ and, as before (with $s=0$), $c_{0}(x,\xi)=c(x)$ and  $m_{\varepsilon}(x,\xi)=e^{\beta x+\varepsilon|x|}$. Note that (\ref{eq2}) is the same as (\ref{asympeq}). Let us first verify that $\psi\in M^{1}_{m_{\varepsilon}}(\mathbb{R})$. In fact, if we take another window $\gamma\in \mathcal{K}_{1}(\mathbb{R})$, we have
\begin{align*}
&\iint_{\mathbb{R}^{2}}|V_{\gamma}\psi(x,\xi)|e^{\beta x+\varepsilon|x|}dxd\xi=\iint_{\mathbb{R}^{2}}(1+|\xi|^{3})|V_{\gamma}\psi(x,\xi)|e^{\beta x+\varepsilon|x|}dx \frac{d\xi}{1+|\xi|^{3}}
\\
&\leq \tilde{C}\left(\iint_{\mathbb{R}^{2}}\psi(t-x)|\gamma(t)|e^{\beta x+\varepsilon|x|}dtdx+\sum_{j=0}^{3} \iint_{\mathbb{R}^{2}}|\psi^{(j)}(t-x)\gamma^{(3-j)}(t)|e^{\beta x+\varepsilon|x|}dtdx\right),
\end{align*}
which is finite (a similar argument shows that $\psi\in M^{1}_{c_{0}}(\mathbb{R})$ if $\int_{-\infty}^{\infty}(\psi (t)+|\psi' (t)|+|\psi'' (t)|)L(e^{|t|})e^{\beta t} dt<\infty$). In view of Theorem \ref{te2}, it is enough to establish  $f\in M^{\infty}_{1/c_{0}}(\mathbb{R})$. Let us first show the crude bound $f(t)=O(c(t))$. Set $A_{1}=\int_{0}^{\infty}\psi(t)dt<\infty$. Since $f$ is non-decreasing, we have
$$
f(x)\leq \frac{1}{A_{1}}\int_{0}^{\infty}f(t+x)\psi(t)dt\leq \frac{1}{A_1}\int_{0}^{\infty}f(t)\psi(t-x)dt\leq A_{2} c(x),
$$
because of (\ref{eq1}) with $\xi=0$. Thus
$$
|V_{\psi}f(x,\xi)|\leq A_{2}\int_{0}^{\infty}c(t)\psi(t-x)dt\leq c(x)\tilde{A}_{\varepsilon}\int_{-\infty}^{\infty}e^{\beta t+\varepsilon|t|}\psi(t)dt< A_{3}c(x), \  \  \ \forall(x,\xi)\in \mathbb{R}^{2}
$$
(likewise in the other case using $L(xy)\leq AL(x)L(y)$), which completes the proof.
\end{proof}

\end{document}